\documentclass[reqno, a4paper]{amsart}
\usepackage{amssymb, amscd, amsfonts,  amsthm}
\usepackage[mathscr]{eucal}
\usepackage{graphicx}
\usepackage[all,poly]{xy} 
 \usepackage[all]{xy} 

\newtheorem{theorem}{Theorem}[section]

\newtheorem{corollary}[theorem]{Corollary}
\newtheorem{proposition}[theorem]{Proposition}

\theoremstyle{definition}

\newtheorem{problem}[theorem]{Problem}

\newcommand{\restrict}{\,{\mathbin{\vert\mkern-0.3mu\grave{}}}\,}

\newcommand{\remove}[1]{}

\DeclareMathOperator{\interval}{[0,1]}

\DeclareMathOperator{\proj}{{\rm proj}}
\DeclareMathOperator{\prim}{{\rm prim}}
\DeclareMathOperator{\range}{\rm range}
\DeclareMathOperator{\maxspec}{\rm maxspec}
\DeclareMathOperator{\Spec}{\rm Spec}
\DeclareMathOperator{\Boole}{\rm Boole}
\DeclareMathOperator{\dist}{\rm dist}

 \title[Centrality of projections]
{Approaching central projections in AF-algebras} 

\author{Daniele Mundici}

\address[D. Mundici]{Department of
Mathematics and Computer Science  ``Ulisse Dini'' \\
University of Florence\\
Viale Morgagni 67/A \\
I-50134 Florence \\
Italy}
\email{ mundici@math.unifi.it }

\date{\today}

\begin{document}



\begin{abstract}  Let $A$ be a unital  
 AF-algebra  whose
Murray-von Neumann order of projections is a lattice.
For any two equivalence classes
$[p]$ and $[q]$ of projections we write 
$[p]\sqsubseteq [q]$
iff    for every primitive ideal  $\mathfrak p$ of $A$ 
either $p/\mathfrak p\preceq q/\mathfrak p\preceq (1- q)/\mathfrak p$
or
$p/\mathfrak p\succeq q/\mathfrak p\succeq (1- q)/\mathfrak p.$
We prove that $p$ is central iff $[p]$ is $\sqsubseteq$-minimal
iff $[p]$ is a characteristic element in $K_0(A)$.
 %
 %
 %
 %
If, in addition,  $A$ is liminary,  then
each  extremal state  of $K_0(A)$ is  discrete,
  $K_0(A)$ has general comparability, and   $A$
 comes equipped with 
a centripetal  transformation
 $[p]\mapsto [p]^\Game$  
that  moves $p$ towards the center.
The number
$n(p) $ of $\Game$-steps needed by $[p]$ to reach the 
center has the monotonicity  property
$[p]\sqsubseteq [q]\Rightarrow n(p)\leq n(q).$
Our proofs combine the $K_0$-theoretic
version of Elliott's classification, the 
categorical equivalence $\Gamma$ between MV-algebras
and  unital $\ell$-groups, 
and \L o\'s ultraproduct theorem
for  first-order logic. 
\end{abstract}

\maketitle

\section{introduction}

Every  C*-algebra $A$ in this paper will be unital and separable.
The  ideals of $A$ will be closed and two-sided.
We let $\proj(A)$ be the set of projections of $A$, 
and  $\prim(A)$ be the
space of primitive  ideals of $A$  with the
Jacobson topology,  \cite[\S 3.1]{dix}.

Following  \cite{bra}, by an  {\it $AF$-algebra} 
we mean the norm closure of the
union of an ascending sequence of finite-dimensional $ C^*$-algebras, all
with the same unit. 

Two projections $p,q$ of  AF algebra $A$ 
are (Murray-von Neumann)
{\it equivalent}, in symbols $p\sim q,$ if
 there is an element $x\in A$ (necessarily a 
partial isometry) such that $p=x^*x$ and $q=xx^*.$
We write $p\preceq q$ if $p$ is equivalent to
a subprojection $r \leq q$. The reflexive and transitive
  $\preceq$-relation is preserved
under equivalence, and   $\preceq$ has the antisymmetry property
$p\preceq q\preceq p\Rightarrow p\sim q,$ because
$A$ is stably finite, 
  \cite[Theorem IV.2.3]{dav}.
The resulting ordering on equivalence classes of
projections in 
$A$ is called the {\it Murray-von Neumann order} of $A$.

Let 
$L(A)$ be the set of equivalence classes $[p]$ of projections 
$p$ of $A$.
Elliott's partial addition in $L(A)$ is defined by setting
 $[p]+[q]$  = $[p+q]$ whenever  $p$ and $q$  are orthogonal. One then obtains a
countable partially
 ordered ``local''  semigroup, which by Elliott's classification
\cite{ell},
is a complete classifier of AF-algebras.  The adjective ``local''
means that the addition operation in $L(A)$ is not always defined.
$L(A)$ inherits a partial order from the  $\preceq$  relation,
and Elliott's partial addition is monotone with respect to this order.

When the Murray-von Neumann order of an AF-algebra $A$
   is a lattice we say that $A$ 
is an  {\it  $AF\ell$-algebra. }

The theory of  AF$\ell$-algebras 
is grounded in  the following result, which will also be
basic for the present paper:

\begin{theorem}
\label{theorem:multi}  
Let $A$ be an AF algebra and $L(A)$  the
Elliott partially ordered  local  semigroup of $A$.
 
\medskip
(i)  \cite{munpan}
Elliott's partially defined addition $+$ in 
$L(A)$   has 
{\em at most one}
 extension to an associative, commutative, 
monotone operation  $\oplus \colon
L(A)^2\to L(A)$ 
satisfying the following condition:
For each projection  $p \in  A$, 
$ [1_A - p ]$   
 is the smallest element $[q]\in L( A)$ with    $[p] \oplus [q]=[1_A].$ 
The
 unique semigroup  $(S(A),\oplus)$
expanding  the Elliott  local  semigroup 
$L(A)$  exists iff    $A$   
is an AF$\ell$-algebra.

\medskip
(ii)  \cite{ell}    Let    $A_1$   and    
$A_2$    be  AF$\ell$-algebras.    
For  each     $j = 1, 2$    let      $\oplus_j$    
be the  extension  of Elliott's addition  given by (i).    
Then the semigroups  $(S(A_1), \oplus_1)$    and   
 $(S(A_2), \oplus_2)$   are isomorphic iff so are      
$A_1$   and    
$A_2$.
 
\medskip
   (iii)   \cite{munpan} For any
   AF$\ell$-algebra 
 $A$   the semigroup
$(S(A), \oplus)$ 
has  the  structure of  a  monoid
$(E(A), 0, \neg,  \oplus)$ 
with an involution operation  
$\neg[p] =  [1_A - p ]$. 
The Murray-von Neumann  
lattice order  of equivalence classes of projections 
    $[p],[q]$  is definable by  the
     involutive monoidal operations of $E(A)$, 
     upon setting
    $[p] \vee [q]=\neg(\neg[p]\oplus [q])\oplus [q]$ and
     $[p]\wedge [q]= \neg(\neg[p]\vee\neg [q])$   
for all $[p],[q] \in  E(A).$

\medskip
(iv)  \cite[Theorem 3.9]{mun-jfa}  
Up to isomorphism, the map 
 $A \to  (E(A),
 0, \neg, \oplus)$ 
   is a one-one correspondence between      AF$\ell$-algebras     and  
countable abelian monoids with a unary operation   $\neg$  satisfying the equations:
$$
\mbox{$\neg\neg x=x,  \,\,\,\,\,\,\,   \neg 0 \oplus x = \neg 0,   \,\,\,\,\,\,
\mbox{ and }   \,\,\,\,\,\,  \neg (\neg x \oplus y) \oplus
 y = \neg(\neg y\oplus x) \oplus x.$}
 $$
These involutive monoids are known as {\em MV-algebras}.
Let\, $\Gamma$ be the categorical equivalence between
 unital  $\ell$-groups and MV-algebras.
Then 
 $(E(A),
 0, \neg, \oplus)$  is isomorphic to 
 $\Gamma(K_0(A))$.
 
 \medskip
   (v)  (From (ii)-(iv) via   \cite{eff}.)\,\,
For any
   AF$\ell$-algebra 
 $A$
the dimension group $K_0(A)$
(which is short for $(K_0(A),K_0(A)^+,[1_A] )$)
is a countable lattice ordered abelian group
with a distinguished strong order unit (for short,
a {\em unital $\ell$-group}).
All countable unital $\ell$-groups arise in this way.
Let    $A$   and    
$A'$    be  AF$\ell$-algebras.
Then $K_0(A)$ and $K_0(A')$
  are isomorphic as unital  
  $\ell$-groups  iff       
$A$   and    
$A'$ are isomorphic.
 \end{theorem}  
 
 We refer to \cite{cigdotmun} and \cite{mun11}
  for background on MV-algebras.
 The following characterization
 (\cite[Proposition 4.13]{mun11}) will find repeated use
 throughout this paper:
 
 \begin{proposition}
 \label{proposition:prime} 
 Let $J$ be an {\em  ideal}  (= kernel
 of a homomorphism) of an MV-algebra $B$. Then the
 following conditions are equivalent:
 \begin{itemize}
 \item[(i)] The quotient  $B/J$ is an {\rm MV-chain},
 meaning that the underlying order of $B/J$ is total.
  \item[(ii)]  Whenever $J$ coincides with the intersection
  of two ideals $H$ and $K$ of $B$, then either $J=H$ or
  $J=K.$
  \end{itemize}
 \end{proposition}

 An ideal of $B$ will be said to be {\it prime} if it
 satisfies the two equivalent conditions above.
 For every MV-algebra $B$ we let
  \begin{equation}
 \label{equation:spec}
  \Spec(B)
 \end{equation} 
 denote the space of prime ideals of $B$ endowed with
 the Zariski (hull-kernel) topology,   (\cite[Definition 4.14]{mun11}).

\begin{corollary}
\label{corollary:spectral}
 In any AF$\ell$-algebra $A$ we have:

\smallskip 
(i)  $K_0$ induces an isomorphism
 $$
\eta \colon   \mathfrak i   
 \mapsto K_0(\mathfrak i)\cap E(A)   
 $$
  between the  lattice of  ideals of $A$ and the lattice of ideals
 of    $E(A)$. Under this isomorphism, 
 primitive ideals of $A$ correspond to 
  prime ideals of $E(A)$.

\medskip
(ii)   The  isomorphism $\eta$ is  a
 homeomorphism of the  space 
 $\prim(A)$ of primitive ideals
 of $A$  with the Jacobson topology,
 onto the space 
$
 \Spec(E(A))
$
 of prime ideals of 
$E(A)$.

\medskip
(iii)  Suppose $I$ is an ideal of the countable MV-algebra $B$.  Let
 the AF$\ell$-algebra  $A$ be  defined   by $E(A)=B$, in view of Theorem \ref{theorem:multi}(iv). 
 Let $\mathfrak i$ be the ideal of $A$ defined by
 $\eta(\mathfrak i)= I$.  Then   
 $ B/I$ is isomorphic to 
 $E(A/\mathfrak i).$

\medskip
(iv)
 For  every ideal $\mathfrak i$ of $A$,  the map
 $$
 \left[\frac{p}{\mathfrak i}\right]\mapsto \frac{[p]}{\eta(\mathfrak i)}\,,
\,\,\,\,\,\,\,\,p\in \proj(A)
$$ 
is an isomorphism
of $E(A/\mathfrak i)$ onto $E(A)/\eta(\mathfrak i).$ 
In particular, for every
$\mathfrak p\in \prim(A)$ the
 MV-algebra $E(A/\mathfrak p)$
is   totally ordered  and  
 $A/\mathfrak p$ 
  has comparability of projections
 in the sense of Murray-von Neumann.  
 
 \medskip
 (v) The map
 $
 J\mapsto J\cap \Gamma(K_0(A))
 $
 is an isomorphism of the lattice of {\em ideals} of $K_0(A)$
  (i.e., kernels of unit preserving
 $\ell$-homomorphisms of $K_0(A)$ into unital $\ell$-groups) 
 onto the lattice of ideals of $E(A).$  Further, 
 $$
 \Gamma\left(\frac{K_0(A)}{J}\right)\cong
  \frac{\Gamma(K_0(A))}{J\cap \Gamma(K_0(A))}\,. 
 $$
\end{corollary}

\begin{proof} (i) From   \cite[Proposition IV.5.1]{dav} and
 \cite[p.196, 21H]{goo-shiva} one gets an isomorphism
 between the lattice of ideals of $A$ and the lattice of
 ideals of the $\ell$-group $K_0(A).$ 
The
preservation properties of $\Gamma$, \cite[Theorems 7.2.2, 
 7.2.4]{cigdotmun} then yield the desired isomorphism.
 For the second statement,  combine
  \cite[Theorem 3.8]{bra} with
  the characterization given in 
Proposition \ref{proposition:prime}
   of prime ideals of an
 MV-algebra. 

\smallskip
(ii)  follows from (i), by definition of the topologies
of $\prim(A)$  and of  $\Spec(E(A))$.

\medskip
(iii)    We have an exact sequence
$$
0\to \mathfrak i\to A\to A/\mathfrak i\to 0.
$$
Correspondingly (\cite[IV.15]{dav}, \cite[Corollary 9.2]{eff}) we have an
exact sequence
$$
0\to K_0(\mathfrak i) \to K_0(A)\to K_0(A/\mathfrak i) \to 0,
$$
whence 
$$
K_0\left(\frac{A}{\mathfrak i}\right)\cong \frac{K_0(A)}{K_0(\mathfrak i)}.
$$
The preservation properties of $\Gamma$ under quotients
 \cite[Theorem 7.2.4]{cigdotmun},
  together with Theorem \ref{theorem:multi}(iv)-(v) yield
  $$
E\left(\frac{A}{\mathfrak i}\right)
\cong  \Gamma\left(K_0\left(\frac{A}{\mathfrak i}\right)\right)\cong
   \frac{\Gamma(K_0(A))}{K_0(\mathfrak i)\cap
   \Gamma(K_0(A))}\cong\frac{E(A)}{\eta(\mathfrak i)}=
   \frac{B}{I}.
  $$
  
\smallskip
(iv)  Combine  (i) and  (iii)  with the preservation properties of
$K_0$ for exact sequences  and the preservation
properties of $\Gamma$ under quotients. 
The   MV-algebra
$E(A/\mathfrak p)$
is   totally ordered by
Proposiiton \ref{proposition:prime}, because, by (ii),
$\eta(\mathfrak p)\in \Spec(E(A))$ 
whenever  $\mathfrak p\in \prim(A)$.
By Theorem \ref{theorem:multi}(iv), 
 $A/\mathfrak p$ 
  has comparability of projections.

\smallskip
(v) This follows by another application of
  \cite[Theorems 7.2.2, 7.2.4]{cigdotmun}.  
 \end{proof}

\bigskip
\section{Central projections in AF$\ell$-algebras}
MV-algebras were invented by C.C.Chang \cite{cha} to give
an algebraic proof of the completeness of the \L ukasiewicz axioms.
For any MV-algebra $D$ we let  
$$
\Boole(D)=\{a\in D\mid a\oplus a=a\}.
$$
As observed by Chang in \cite[Theorems 1.16-1.17]{cha},
${\Boole}(D)$  is a subalgebra of $D$ which turns out to be
a boolean algebra.

The following theorem and its
extension Theorem \ref{theorem:extension} make precise the
 intuition that commutative AF-algebras stand to boolean
algebras as AF$\ell$-algebras stand to MV-algebras:

\begin{theorem}
\label{theorem:threeconditions}
For every projection $p$ of an 
  AF$\ell$-algebra $A$
the  following conditions are equivalent:

\smallskip
\begin{itemize}
\item[(i)] $p/\mathfrak p \in\{0,1\}\subseteq A/\mathfrak p$\,\,\,
 for all \,\,$\mathfrak p\in \prim(A)$.

\medskip
\item[(ii)]  $[p]\in {\Boole}(E(A))$.

\medskip
\item[(iii)]  $p$ is central in $A$.

\medskip
\item[(iv)] $[p]$ is a {\em characteristic element} of
$K_0(A),$ in the sense that 
 $[p]\wedge [1_A-p]$ exists and equals $0$,
 \cite[Definition p.127]{goo-ams}. 
\end{itemize}
\end{theorem}

\begin{proof} 
(ii)$\Rightarrow$(i)  From the assumption
$[p]\oplus[p]=[p]$
 it follows that
 $[p]/P\oplus [p]/P=[p]/P$, whence 
$[p]/P\in \Boole(E(A)/P) =  \{0,1\}$
 for each 
 $P\in \Spec(E(A))$, because  $E(A)/P$
 is totally ordered  (Proposition  \ref{proposition:prime}).   Let
$\mathfrak p$ be the primitive ideal of $A$ given by
$\eta(\mathfrak p)=P$, with
  $\eta$ the isomorphism of 
  Corollary \ref{corollary:spectral}(i).
Then $[p/\mathfrak p]\in \Boole(E(A/\mathfrak p))=\{0,1\},$
because
by
Corollary \ref{corollary:spectral}(iv),
 $E(A/\mathfrak p)\cong E(A)/P$ is totally ordered.
Since $P$ is an arbitrary prime ideal of $E(A)$, 
then $\mathfrak p$ is an
arbitrary primitive ideal of   $A$. We conclude that 
$p/\mathfrak p \in \{0,1\}$ for all  $\mathfrak p\in \prim(A)$.
  
\smallskip
(i)$\Rightarrow$(iii) The hypothesis implies that $p/\mathfrak p$
is central in $A/\mathfrak p$ for each  
 $\mathfrak p\in \prim(A)$.
Then $p$ is central in $A$.

\smallskip
(iii)$\Rightarrow$(ii)
  By way of contradiction assume $p$ central in $A$ but
  $[p]\notin {\Boole}(E(A))$.  By
  \cite[Corollary 1.2.14]{cigdotmun},
  $$
  \bigcap\{P\mid P\in \Spec(E(A))\}=\{0\}.
  $$
Therefore, 
   for some 
  $P\in \Spec(E(A))$, $[p]/P$ does not belong to $ \Boole(E(A)/P)$.
In view of Corollary \ref{corollary:spectral}(i), let
$\mathfrak p\in \prim(A)$ be defined by 
$\eta(\mathfrak p)=P.$
Then 
\begin{equation}
\label{equation:77}
  \frac{p}{\mathfrak p}\,\, \mbox{ is central and } \,\, \frac{p}{\mathfrak p}\notin \{0,1\} \subseteq  \frac{A}{\mathfrak p}.
\end{equation}

\medskip
\noindent{\it Claim.}   The projections
 ${p}/{\mathfrak p}$ and $(1_A-p)/\mathfrak p$
 of $A/\mathfrak p$
 are not Murray-von Neumann comparable. 

Arguing by way of contradiction, let us suppose
 ${p}/{\mathfrak p}$ and $(1_A-p)/\mathfrak p$
  are comparable, say,
$$
u^*u=\frac{1_A-p}{\mathfrak p}\,\,\ 
\mbox{ and }\,\,\,uu^*=\frac{q}{\mathfrak p}\leq \frac{p}{\mathfrak p},
$$
 for some
partial isometry $u\in A/\mathfrak p$ and
 ${q}/{\mathfrak p}\in \proj(A/\mathfrak p).$
From
$$\frac{q}{\mathfrak p}\,\,\frac{p}{\mathfrak p}=\frac{p}{\mathfrak p}\,\,
\frac{q}{\mathfrak p}=\frac{q}{\mathfrak p}$$
it follows that 
$$
\frac{1_A-p}{\mathfrak p}=\frac{1_A-p}{\mathfrak p}\,\,\,\frac{1_A-p}{\mathfrak p}=u^*uu^*u=u^*\frac{q}{\mathfrak p}u\leq u^*\frac{p}{\mathfrak p}u=u^*u\frac{p}{\mathfrak p}=\frac{1_A-p}{\mathfrak p}\,\,\frac{p}{\mathfrak p}=0\,,
$$
against \eqref{equation:77}. 
Our claim is settled.

%
%
%

\medskip

On the other hand, since $P$ is prime, then 
   $E(A)/P$ is totally
ordered,  by Proposiiton \ref{proposition:prime}.
 By 
  Corollary \ref{corollary:spectral}(iii)-(iv),
   $A/\mathfrak p$ has comparability, which contradicts
   our claim.

\medskip
(ii) $\Leftrightarrow$ (iv). 
  By Theorem \ref{theorem:multi}(iv)-(v)  and
definition of $\Gamma$  (\cite[Definition 2.4]{mun-jfa}), 
we can write
$E(A)=\{x\in K_0(A)\mid 0\leq x\leq u\}$, where
$u$ is the order-unit of $K_0(A)$, coinciding with
the unit element 1$=[1_A]$ of $E(A).$ By \cite[Theorem 2.5]{mun-jfa}
the lattice order of $E(A)$ agrees with the restriction
to $E(A)$ of the
lattice order of $K_0(A).$
The desired conclusion now follows from
  \cite[Theorem 8.7, p.130]{goo-ams}.
\end{proof}

 \bigskip
\subsection*{The ordering $\sqsubseteq$ and the map 
$\sigma^*\colon\interval\to \interval$}
The rest of this section is devoted to proving the
following extension of 
Theorem \ref{theorem:threeconditions}:

\begin{theorem}
\label{theorem:extension}
Assume $A$ is an  AF$\ell$-algebra.  
For any 
 $x,y\in E(A)$ let us write
 $x\sqsubseteq y$\,\, iff \,\, for every  prime ideal $P$ of $E(A)$
 $$
(y/P<\neg y/P\,\,\, \mbox{implies} \,\,\,  x/P\leq y/P)
    \,\,\,\,\mbox{  and  }\,\,\,
 (y /P>\neg y/P \,\,\, \mbox{implies}  \,\,\,   x /P\geq y/P),
$$
with $\leq$ the underlying total order of $E(A)/P$,
(Corollary  \ref{corollary:spectral}(i)).
Then 
$\sqsubseteq$ endows 
$E(A)$
with  a partial order (reflexive, transitive,
antisymmetric) relation.  
Further, for every  $p\in \proj(A)$,
the equivalent  conditions (i)-(iv) in Theorem \ref{theorem:threeconditions}  are equivalent to
$[p]$ being  $\sqsubseteq$-minimal in $E(A).$
\end{theorem}

\medskip
The following transformation will play a key role in the sequel:
 Let $\tau=\tau(X)$ be an MV-term in 
the variable $X$, \cite[Definition 1.4.1]{cigdotmun}.
For any  MV-algebra $B$ and  $a\in B$,
by induction on 
  the number of operation symbols in $\tau$ let us define  
\begin{equation}
\label{equation:term}
a_X=a, \,\,\,a_{\tau_1\oplus \tau_2}=a_{\tau_1}\oplus a_{\tau_2},
\,\,\,\,a_{\neg \tau}=\neg a_\tau.
\end{equation}
This
transforms $a$ into an element  $a_\tau\in B.$ 
The ambient algebra $B$ will always be clear from the context.
For every ideal $I$ of $B$, induction on the number
of operation symbols in $\tau$ yields 
\begin{equation}
\label{equation:quotient}
(a/I)_\tau= a_\tau/I.
\end{equation}

Following
\cite[p.8]{cigdotmun}, for any two MV-terms  $\rho, \tau$ we 
let
$
\rho\odot \tau$
denote the MV-term 
$
\neg(\neg \rho
\oplus \neg \tau).
$
Correspondingly, for any two elements
$a,b$ of an MV-algebra $B$ we write
\begin{equation}
\label{equation:odot}
a\odot b \mbox{ as an abbreviation of } \neg(\neg a
\oplus \neg b).
\end{equation}

Let   $Free_1$ denote  the
free one-generator MV-algebra.
As a special case of McNaughton representation theorem,
 (\cite[Corollary 3.2.8, Theorem 9.1.5]{cigdotmun}),  
$Free_1$ is the MV-algebra of all {\it
one-variable  McNaughton functions},
 those continuous piecewise  (affine)  linear functions
$f\colon \interval\to \interval$ 
whose linear pieces have integer
coefficients. Further,
the identity function
$\pi_1\colon\interval\to\interval$ freely
generates $Free_1$.

 \begin{figure}
    \begin{center}
    \includegraphics[height=8.96cm]{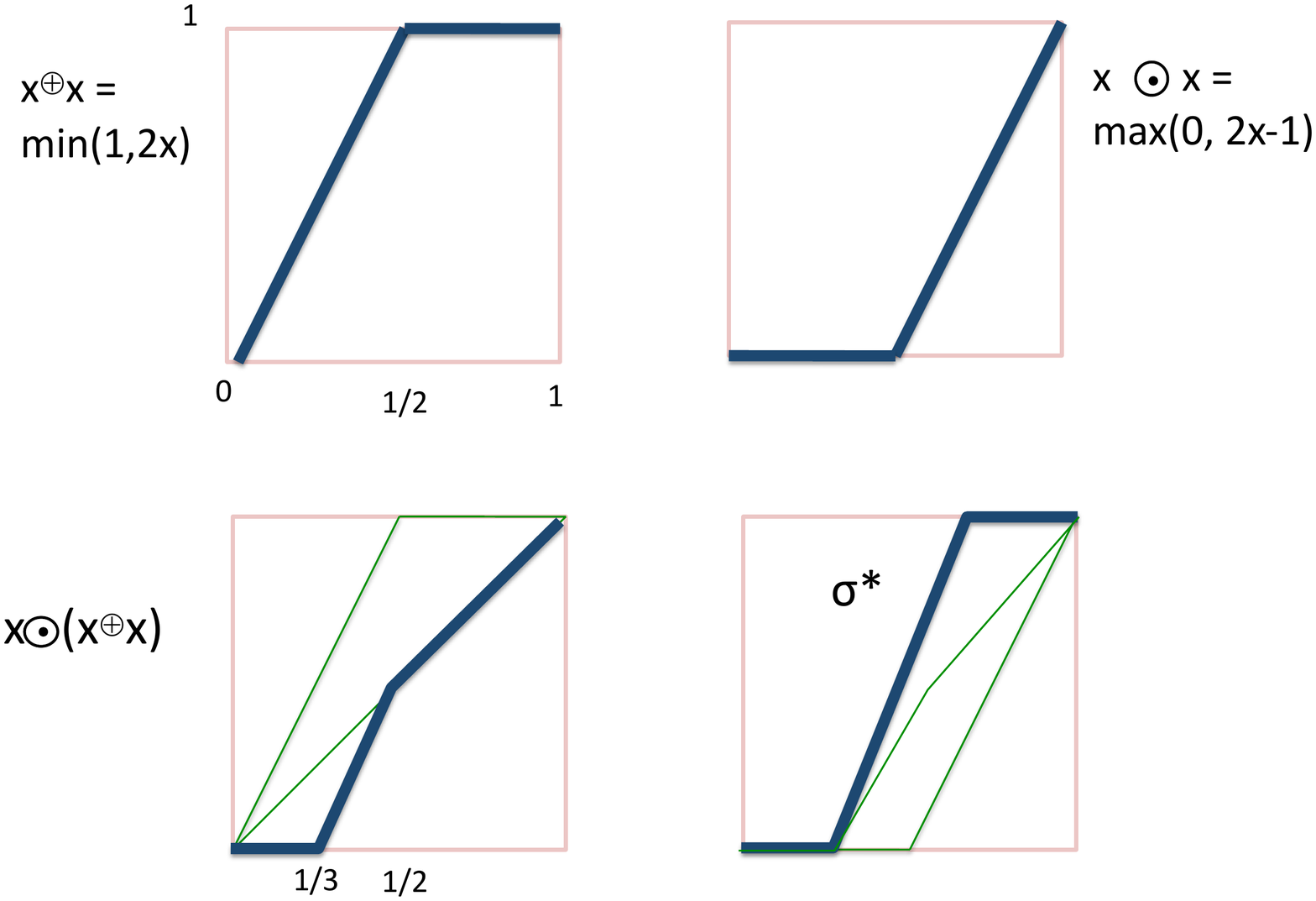}
    \end{center}
        \caption{\small  
        The graph of the function 
$\sigma^* (x)=(x\odot(x\oplus x))
\oplus (x \odot x)=
\min(1,\max(0,3x))$, and of some of its constituents.
$\sigma^*\colon\interval\to\interval$ is a member of the
free one-generator MV-algebra $Free_1$ 
consisting of all
one-variable McNaughton functions.
As usual,   $x\odot x$
is an abbreviation of $\neg(\neg x\oplus \neg x)$.
}
    \label{figure:spezzate}
   \end{figure}

\begin{proposition}
\label{proposition:transformation}
With the notation of
\eqref{equation:term}
and  \eqref{equation:odot},
let  the MV-term $\sigma$ be defined  by
$\sigma =(X\odot(X\oplus X))
\oplus (X \odot X).$
Let us write
  $\sigma^*$ instead of $(\pi_1)_\sigma.$

\medskip
\begin{itemize}
\item[(i)]
For all $x\in \interval,$
$
\sigma^*(x)
=(x\odot(x\oplus x))
\oplus (x \odot x)=\min(1,\max(0,3x)).
$

\medskip
\item[(ii)]
More generally, for any cardinal
$\kappa>0,$
let  $f $ belong to  the
free $\kappa$-generator MV-algebra  $Free_\kappa$  (the algebra of
  McNaughton functions over the Tychonoff cube
  $\interval^\kappa$, \cite[Theorem 9.1.5]{cigdotmun}).
Then $f_\sigma= \sigma^*\circ f$, with $\circ $ denoting composition.
\end{itemize}
\end{proposition}
\begin{proof} (i) A routine verification. 
See Figure \ref{figure:spezzate}.
(ii)  By  induction on the number of operation symbols in 
an MV-term coding $f$.
\end{proof}

Following tradition, 
 by  the  {\it standard} MV-algebra  $[0,1]$ we mean
the real unit real interval equipped
by  the operations $\neg y=1-y$ and $y\oplus z=
\min(1,y+z)$. There will never be danger of confusion
between the standard MV-algebra and the real 
unit interval.

\begin{proposition}
\label{proposition:sigma}
  For any MV-algebra $D$ and    $c \in D$,
$c_\sigma \sqsubseteq c.$ 
\end{proposition}

\begin{proof}

Consider the conjunction  $\psi$ of the following  statements
in the language of MV-algebras:
 $$
\forall  z\,\,\, 
(\mathsf{IF} \,\,\, (z  < \neg z)
\,\,\mathsf{THEN}\,\,\, 
(z \odot(z \oplus z))
\oplus (z \odot z)\leq z) 
$$
$$
\forall  z\,\,\, 
(\mathsf{IF} \,\,\, (z  > \neg z)
\,\,\mathsf{THEN}\,\,\, 
(z\odot(z\oplus z))
\oplus (z \odot z)\geq z), 
$$
in first-order logic with the usual connectives,
quantifiers and identity.
Here $\leq $ is the natural order of any MV-algebra:
$a\leq b$ iff $\neg a\oplus b=1,$ \cite[p.9]{cigdotmun}.
 $\psi$ is
satisfied by  the standard MV-algebra.
This is so because $\psi$ says
 $
\forall  z\,\,\,  (z\odot(z\oplus z))
\oplus (z \odot z) \sqsubseteq z. 
$ 
i.e.,
$
\forall  z\,\,\, z_\sigma\sqsubseteq z,
$
which is easily verified in $\interval$.

By \L o\'s ultraproduct theorem \cite[Theorem 4.1.9, Corollary 4.1.10]{chakei},   $\psi$ is satisfied by  every ultrapower  
$\interval^\ast$ of the standard MV-algebra.

By Di Nola's representation theorem, 
 \cite[9.5.1]{cigdotmun},
every  MV-chain  is embeddable in some
ultrapower of $\interval$. 
Since $\psi$ is a universal
sentence, we have thus shown:
\begin{equation}
\label{equation:dinola}
\mbox{for each element $z$ of every MV-chain,}\,\,\,
 z_\sigma\sqsubseteq z.
\end{equation}
 
\smallskip
To conclude the proof, arguing by
  way of contradiction, suppose  there is an MV-algebra $D$
and $c\in D$ such that $c_\sigma\sqsubseteq c$ fails.  
By definition, there is  $P \in \Spec(D)$ such that  in the
quotient MV-chain $D/P$,\,\, (Proposiiton \ref{proposition:prime}), 
 we either have
$c/P<\neg c/P$ and   $c_\sigma/P> c/P$, or
$c/P>\neg c/P$ and   $c_\sigma/P < c/P$.   Say 
without loss of generality, 
$c/P<\neg c/P$ and   $c_\sigma/P> c/P$.   
By \eqref{equation:quotient}, 
$c_\sigma/P= (c/P)_\sigma.$
So $c/P$  is a counterexample of
 \eqref{equation:dinola} in
  $D/P$, and the proof is complete.  
\end{proof}

\begin{proposition}
\label{proposition:sigma-bis}
 Every  MV-chain $C$  
  satisfies the conjunction 
   of the following two sentences
 of first-order logic:
 $$
 \forall x\,\,\,
 (\mathsf{IF}\,\,\,0<x<\neg x \,\,\,\mathsf{THEN}\,\,\, 
 (x\odot(x\oplus x))
\oplus (x \odot x) < x)
 $$
  $$
 \forall x\,\,\,
 (\mathsf{IF}\,\,\, \neg x <x<1 \,\,\,\mathsf{THEN}\,\,\,
 (x\odot(x\oplus x))
\oplus (x \odot x) > x).
 $$ 
\end{proposition}

\begin{proof} It is enough to deal with the first sentence, 
denoted  $\chi$.
First of all, observe that $\chi$ is  a universal
sentence of first-order logic in the language of MV-algebras:
Thus, ``$0 < x$'' means  ``$\,\,\mathsf{NOT}\,(0=x)$''. 
 Also,  ``$x<\neg x$'' means
``$(x\leq \neg x)\,\,\mathsf{AND}\,\,\mathsf{NOT}\,\,\,(x=\neg x)$'', i.e.,
 ``$(\neg x \oplus \neg x=1)
\,\,\mathsf{AND}\,\,\mathsf{NOT}\,\,\,(x=\neg x)$''.  Similarly, ``$x_\sigma <x$'' means
``$\mathsf{NOT}(x_\sigma = x)\,\,\,\mathsf{AND}\,\,\,(\neg x_\sigma \oplus x=1) $''.
As we have seen, the MV-term
$x_\sigma$ is definable from $x$ and the MV-algebraic
operations.
Arguing as in the proof of Proposition
\ref{proposition:sigma},  
$\chi$  is satisfied by the standard
MV-algebra. By \L o\'s theorem, $\chi$ is satisfied by
any ultrapower $\interval^*$, whence it is satisfied
by   $C$, because $C$ can be embedded into
some ultrapower of $\interval$, by Di Nola's theorem.
\end{proof}

\begin{proposition}
\label{proposition:order}
For any MV-algebra $B$,
 $\sqsubseteq$ is a 
partial order relation on $B$.
\end{proposition}

\begin{proof} Reflexivity is trivial. To verify transitivity,
let us assume $x\sqsubseteq y\sqsubseteq z$ but
$x\sqsubseteq z $ fails (absurdum hypothesis).
There is a prime ideal $P$ of $B$ such that, without loss of
generality, $x/P < \neg x/P$ but $z/P\nleq x/P$, whence 
$$z/P > x/P,$$
because $B/P$ is totally ordered,
by Proposiiton \ref{proposition:prime}. 
 From   $x\sqsubseteq y$ we have  $y/P\leq x/P $.
Thus,  by the contrapositive property,
  (\cite[Lemma 1.1.4(i)]{cigdotmun}), 
$\neg y/P \geq \neg x/P>x/P\geq y/P$. 
From $y\sqsubseteq z$ we now get
$z/P\leq y/P\leq x/P$, a contradiction.

\smallskip
To check the antisymmetry property, suppose
\begin{equation}
\label{equation:anti}
x\sqsubseteq y\sqsubseteq x
\end{equation}
 but $x\not=y$, (absurdum hypothesis).
Thus  $\dist(x,y)\not=0$, where   
\begin{equation}
\label{equation:distance}
\dist(x,y)=(x\odot\neg y)\oplus(y\odot \neg x),
\end{equation}
is  Chang's {\it distance function},
  \cite[p.477]{cha}, \cite[Definition 1.2.4]{cigdotmun}.
By  Proposiiton \ref{proposition:prime} there is
a prime ideal $P$ of $B$ such that   $\dist(x/P,\,  y/P)\not=0$, i.e.,
$x/P\not=y/P.$ We now argue by cases:

\smallskip
If   $x/P<\neg x/P$ and $y/P<\neg y/P$
then from \eqref{equation:anti} we obtain
$x/P\leq y/P\leq x/P$, whence  $x/P=y/P$,
a contradiction.

\smallskip
If   $x/P>\neg x/P$ and $y/P>\neg y/P$ 
we similarly obtain a contradiction with  $x/P\not=y/P.$

\smallskip
If  $x/P < \neg x/P$ and
$y/P>\neg y/P$, combining 
  \eqref{equation:anti}  with
  the contrapositive property
   \cite[Lemma 1.1.4(i)]{cigdotmun},
  we obtain
$y/P\leq x/P < \neg x/P \leq \neg y/P$, whence 
$y/P<\neg y/P$, which is impossible. 

\smallskip
If  $x/P=\neg x/P$ and $y/P=\neg y/P$ 
then an easy verification similarly shows that
$x/P=y/P$, another contradiction.

\smallskip
Without loss of generality  the last possible case is  
$x/P=\neg x/P$ and $y/P< \neg y/P$.
Then by \eqref{equation:anti},  $x/P\leq y/P$, whence
$\neg x/P=x/P \leq y/P< \neg y/P$.   Again by
  the contrapositive property 
\cite[Lemma 1.1.4(i)]{cigdotmun}, the two inequalities
$x/P \leq y/P$ and $\neg x/P < \neg y/P$ are
contradictory.

\smallskip
 Having thus obtained a contradiction in all possible cases,
we have completed  the proof.
\end{proof}

\subsection*{End of the proof of Theorem \ref{theorem:extension}}

Trivially,  every   $b\in \Boole(E(A))$ satisfies
$b/P\in \{0,1\}\subseteq E(A)/P$ for every prime ideal 
$P$ of $E(A)$. Thus 
$b$ is  $ \sqsubseteq$-minimal. 
Conversely,  for any element $b$ of   $E(A)$
we will prove  
\begin{equation}
\mbox{If}\,\,\, 
b\notin \Boole(E(A))\mbox{\,\,\,then\,\,\,} b \mbox{ is not 
 $\sqsubseteq$-minimal.}
\end{equation}
By way of contradiction assume $b\notin \Boole(E(A))$ and
$b$ is $\sqsubseteq$-minimal. 
Following   \cite[Definition 4.14]{mun11},  
for any MV-algebra $B$  let
\begin{equation}
\label{equation:maximal}
\boldsymbol{\mu}(B)
\end{equation}
denote the maximal spectral space of $B$
equipped  with the hull-kernel (Zariski) topology inherited from
$\Spec(B)$ by restriction.
By \cite[Proposition 4.15]{mun11},\,\, $\boldsymbol{\mu}(B)$
is a nonempty  compact Hausdorff subspace
of the prime spectral space $\Spec(B)$.
For each $M\in \boldsymbol{\mu}(B)$ there is
a unique embedding of $B/M$ into the standard
MV-algebra  $\interval$, \cite[Theorem 4.16]{mun11}.
So for each  $a\in B$ there is a unique
$\alpha \in \mathbb R$ such that  $a/M = \alpha$.
We will throughout identify  $a/M$ and $\alpha$ without
fear of ambiguity.

\medskip
\noindent
{\it Claim 1:}  For all  $M\in \boldsymbol{\mu}(E(A))$
we have $b/M\in \{0,1/2,1\}$.

By way of contradiction, suppose  $0<b/M<1/2.$  (The case
$1/2<b/M<1$ is similar).  By Proposition \ref{proposition:sigma-bis},
in the MV-chain  $E(A)/M\subseteq \interval$ we have 
$b_\sigma/M < b/M$, whence  $b_\sigma\not =b.$ On the other hand,
by Proposition \ref{proposition:sigma}, $b_\sigma \sqsubseteq b.$
It follows that $b$ is not $\sqsubseteq$-minimal, a contradiction.

\medskip
\noindent
{\it Claim 2:} For all  $M\in \boldsymbol{\mu}(E(A))$, if
$b/M=0$ then $b/P=0$ for all $P\in \Spec(E(A))$ 
contained in $M$.

Otherwise (absurdum hypothesis), there is a
maximal ideal $M$ with  $b/M=0$,  and  a prime
ideal $P\subseteq M$ with $b/P>0$. 
Thus
$$
0 < b/P<\neg b/P.
$$
(For otherwise $b/P\oplus b/P=1$ whence a fortiori
$b/M\oplus b/M=1$, and $b/M\geq 1/2$, which is impossible.)  
By Proposition \ref{proposition:sigma-bis},
in  the MV-chain  $E(A)/M\subseteq \interval$
we have
 $
(b/P)_\sigma < b/P,
$
 whence by  \eqref{equation:quotient}, 
$b_\sigma/P=  (b/P)_\sigma \not=  b/P
$
whence
$b_\sigma\not=b.$
By Proposition \ref{proposition:sigma},
$b_\sigma\sqsubseteq b$,
thus contradicting  the $\sqsubseteq$-minimality of $b$.

\medskip
Similarly,

\medskip
\noindent
{\it Claim 3:} For all  $M\in \boldsymbol{\mu}(E(A))$, if
$b/M=1$ then $b/P=1$  for all $P\in \Spec(E(A))$ 
contained in $M$.

\medskip
\noindent
{\it Claim 4:} For all  $M\in \boldsymbol{\mu}(E(A))$, if
$b/M=1/2$  (i.e., $b/M=\neg b/M$) then 
$b/P= \neg b/P)$,
 for all $P\in \Spec(E(A))$ 
contained in $M$.

\smallskip
Otherwise (absurdum hypothesis), there is a 
maximal ideal $M$ and a prime ideal $P\subseteq M$
with $b/M=1/2$ and $b/P\not= \neg b/P$, say without loss
of generality
$b/P<\neg b/P$ in the MV-chain $E(A)/P$.
If  $b/P=0$, i.e., if $b\in P$, then 
$b\in M$ whence $b/M=0$, which is impossible.  So
$b/P>0$.  By Proposition \ref{proposition:sigma-bis}, in
the  MV-chain   $E(A)/P$ we have     
$
(b/P)_\sigma < b/P,
$
whence
$
b_\sigma/P=  (b/P)_\sigma \not=  b/P
$
and
$b_\sigma\not=b.
$
By  Proposition \ref{proposition:sigma},
$b_\sigma\sqsubseteq b$,
again contradicting  the $\sqsubseteq$-minimality of $b$.

\smallskip
We have thus shown that   every prime ideal $P$ of $E(A)$
belongs to precisely one of the following three sets:
$$
Y_0=\{P\in \Spec(E(A)) \mid  b\in P\},\,\,\,
Y_1=\{P\in \Spec(E(A)) \mid  \neg b\in P\}
$$
$$
Y_{1/2}=\{P\in \Spec(E(A)) \mid 
b/P=\neg b/P, \mbox{ i.e.,  }\dist(b/P, \neg b/P)\in P \},
$$
where
$
\dist(x,y)=(x\odot\neg y)\oplus(y\odot \neg x)
$
is  Chang's distance function, (see \eqref{equation:distance}).

To conclude, let  $c\in E(A)$ be defined by
$
c= b\odot b = \neg(\neg b\oplus \neg b).
$
For each  $P\in \Spec(E(A))$ the element  
$c/P= b/P\odot b/P$  equals 1 if
$b/P=1$, equals 
0 if $b/P=0,$ and 
equals 0 if   $b/P = \neg b/P.$
 It follows that $c\sqsubseteq b,$ because, as we
have just seen, the prime quotients of  $b/P$ 
have no other possibilities.
Our hypothesis 
$b\notin \Boole(E(A))$ implies $Y_{1/2}\not=\emptyset$,
whence there is
prime ideal $R$ of $E(A)$ such that $b/R=\neg b/R$.
Since  $c/R=0$, then  $c\not= b,$ a contradiction with
the  $\sqsubseteq$-minimality of
$b$. 
The  proof of Theorem \ref{theorem:extension} is complete.
 \hfill $\Box$  
 
\bigskip
 
    Figure \ref{figure:example} is an illustration of the
    equivalence classes   $[p]$ 
   and   $[p]_\sigma$ for $p$ a projection
   in the
   AF$\ell$-algebra  $\mathfrak M_2$
   defined by $E(\mathfrak M_2)$  = the free
   two-generator MV-algebra  $Free_2$ consisting of all
   McNaughton functions over the unit real square $\interval^2.$
  \,\,\, $\mathfrak M_2$ is well defined by Theorem  \ref{theorem:multi}(iv).    
As  $n$ tends to $\infty$, letting
   $\varsigma(n)=\sigma\circ\dots\circ\sigma
   \,\,\,\,(n\,\, {\rm times})$ the grey zone in 
    $[p]_{\varsigma(n)}$ gets thinner and thinner,
    and the density plot of   $[p]_{\varsigma(n)}$
   is almost everywhere white or black.

     \begin{figure}
    \begin{center}
    \includegraphics[height=9.96cm]{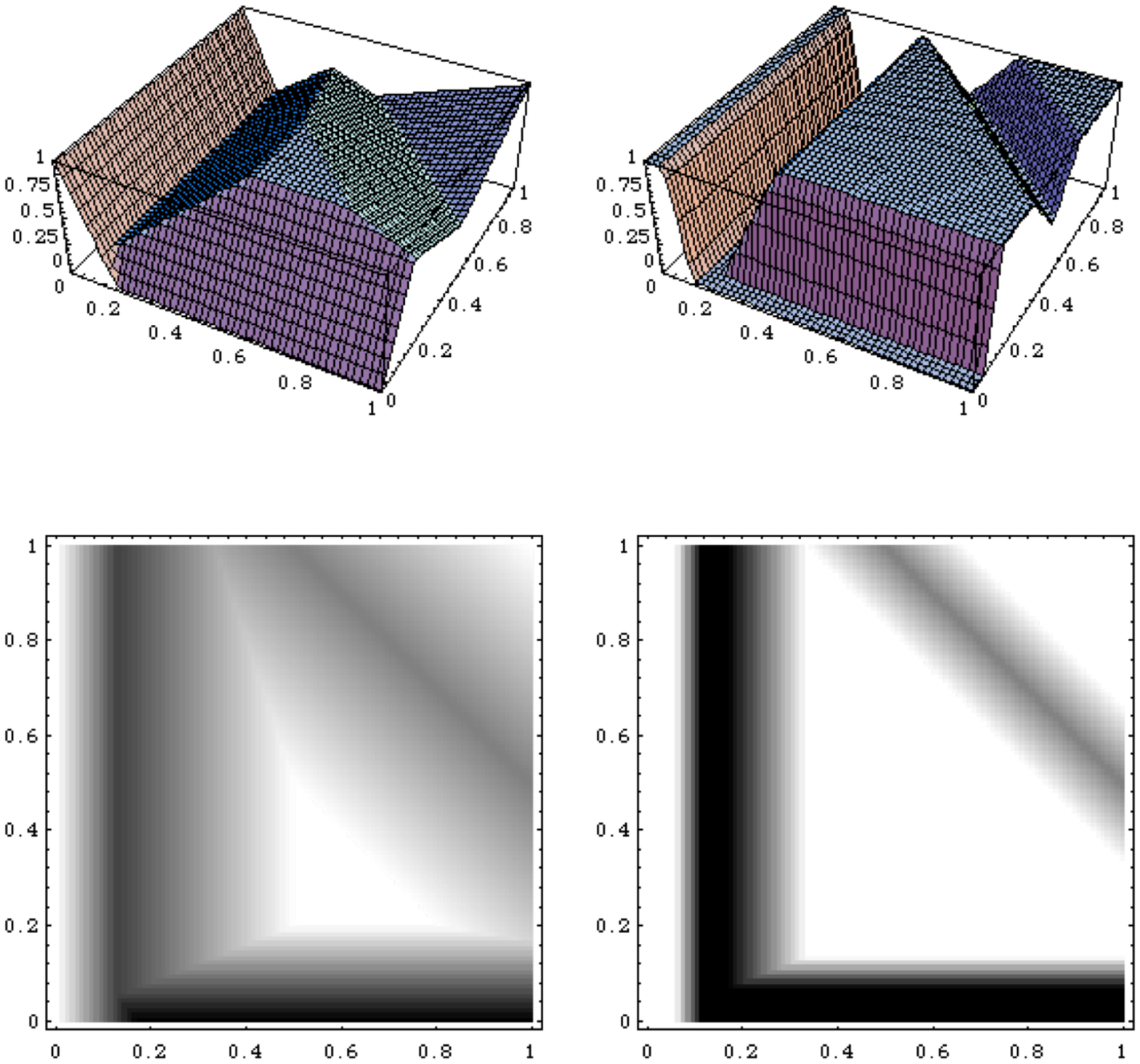}
    \end{center}
        \caption{\small  Left column: the graph and the
        density plot of
        the  Murray-von Neumann equivalence class
  $[p]$ of a projection
   of  the AF$\ell$-algebra  $\mathfrak M_2$
   defined by $E(\mathfrak M_2)$ = the free
   two-generator MV-algebra  $Free_2$. 
Right column: the graph and the density plot of 
the transformed equivalence class   
$[p]_\sigma\in E(\mathfrak M_2)$.
}
    \label{figure:example}
   \end{figure}

\section{The special case of 
liminary C*-algebras with boolean spectrum}
In this section we consider a class of C*-algebras
whose central projections have a particularly
simple  realization.

As the reader will recall, a
totally disconnected
compact Hausdorff
space is said to be {\it boolean}. 

\begin{theorem}
\label{theorem:equivalent-conditions} 
For any  (always unital and separable)
liminary C*-algebra $A$ the following conditions
are equivalent:
\begin{itemize}
\item[(i)]  $A$ has a boolean (primitive) spectrum.

\smallskip
\item[(ii)] $A$ is an AF$\ell$-algebra.
\end{itemize}
\end{theorem} 

\begin{proof}
(i)$\Rightarrow$(ii)   From 
 \cite[Step(i), p.80]{braell} it follows that 
$A$ is an AF-algebra. Now by
  \cite[Theorem 1]{ellmun},
  $K_0(A)$ is lattice-ordered.
  Finally,  by Theorem \ref{theorem:multi}(iv)-(v), 
 $A$ is an AF$\ell$-algebra.

\smallskip
(ii)$\Rightarrow$(i)  
  Since $A$ is liminary, all  its primitive ideals are
  maximal,  \cite[4.1.11(ii), 4.2.3]{dix}.
   Since $A$ is an AF$\ell$-algebra, 
   by Corollary \ref{corollary:spectral}(i)
every prime ideal  of  $E(A)$ is 
 maximal.   In symbols, by
 \eqref{equation:spec} and \eqref{equation:maximal},
 \begin{equation}
 \label{equation:prime=maximal}
 \Spec(E(A))=\boldsymbol{\mu}(E(A)).
 \end{equation}
 This is a necessary and sufficient condition
 for $E(A)$ to be {\it hyperarchimedean}, 
 \cite[Theorem 6.3.2]{cigdotmun}. Since
  the intersection of all prime ideals of $E(A)$
  is zero, (\cite[Corollary 1.2.14]{cigdotmun}),
  then   $E(A)$ is {\it semisimple}, \cite[p.72]{cigdotmun}.
By   \cite[Proposition 1.2.10]{cigdotmun}, for
 every  prime ideal $P$ of $E(A)$, the quotient 
 $E(A)/P$ has no nonzero ideals,  because
 $P$ is maximal.  Equivalently,  
 $E(A)/P$ is {\it simple},  \cite[Theorem 3.5.1]{cigdotmun}.
By \cite[Proposition 4.15]{mun11},\,\,\, $\boldsymbol{\mu}(E(A))$
is a nonempty  compact Hausdorff space.
By  \cite[Theorem 4.16]{mun11}, $E(A)$ is isomorphic
to a separating MV-algebra of continuous  
$\interval$-valued functions 
on $\boldsymbol{\mu}(E(A))$.
Since  $E(A)$ is  
   hyperarchimedean, 
   from   \cite[Corollary 6.3.5]{cigdotmun}
   it follows that
$\boldsymbol{\mu}(E(A))$  is a {\it boolean space.}  
By Corollary \ref{corollary:spectral}(ii),
  $\prim(A)$ is boolean. 
\end{proof}
   
 The following theorem provides a useful 
 representation of $E(A)$ as an MV-algebra of
 continuous rational-valued functions over $\prim(A)$:
  	 
\begin{theorem}
\label{theorem:equivalence-unified}   
Suppose the liminary C*-algebra
$A$ satisfies the two equivalent conditions  of Theorem
\ref{theorem:equivalent-conditions}.
For every projection $q$ of  $A$ let
the  {\em  dimension map }
$d_q: \prim(A) \rightarrow \mathbb Q\cap \interval$  be defined by  
\begin{equation}
\label{equation:dimension}
d_q(\mathfrak p) =  \frac{\dim\,\,\range \,\, \pi(q)}
{\dim \pi}\,\,\,\,(\mathfrak p\in \prim(A)),
\end{equation}
where $\pi$   is an arbitrary
 irreducible representation 
of $ A$  such that   ker$(\pi) = \mathfrak p.$
 
\begin{itemize}

\item[(I)]  
The map   
 $
  [q]\in E(A)\mapsto d_q\in  \interval^{\prim(A)}
 $  
is an isomorphism of   $E(A)$ onto 
the  MV-algebra   of  dimension
maps of $A$, with the pointwise operations 
of the standard
MV-algebra $\interval$.

\smallskip
\item[(II)]
Each  dimension map  is 
continuous  and has  a finite range.

\smallskip
\item[(III)] (Separation)
For any two distinct $\mathfrak p,\mathfrak q\in \prim(A)$
there is  $p \in \proj(A)$ such that $d_p(\mathfrak p)=0$ and
$d_p(\mathfrak q)=1$.   (Equivalently,  
there is $r\in \proj(A)$ with $d_r(\mathfrak p)\not=
d_r(\mathfrak q).$)

\end{itemize}
\end{theorem}

\begin{proof}
Any  two irreducible representations of $A$
 with the same kernel are 
 equivalent  (\cite[Theorem 4.3.7(ii)]{dix}),  and
 finite-dimensional
 (\cite[4.7.14(b)]{dix}).   Thus  the  
actual choice 
of  the representation $\pi$
with    ker$(\pi) = \mathfrak p$
is  immaterial  in \eqref{equation:dimension}, and 
the   dimension map 
$d_q$ is well defined. 
For each  $\mathfrak p\in \prim(A)$ and irreducible representation
  $\pi$ of $A$ with $\ker \pi=
\mathfrak p$,  the quotient $A/\mathfrak p$
is simple, because  $\mathfrak p$ is maximal. So upon setting
$d=\dim \pi$ we  have
\begin{equation}
\label{equation:uno}
A/\mathfrak p \cong M_d, \mbox{ the C*-algebra of $d\times d$ complex matrices}.
\end{equation}

\medskip
(I)  We first show that
$d_q$ depends on $q$ only via its Murray-von Neumann
equivalence class $[q]$.
For the proof we prepare:

  \medskip
  \noindent
  {\it Claim 1:}
   For any  $p, q\in \proj(A)$,\,\,\,     
$$p\sim q\,\,\,\, \mbox{ iff }\,\,\,\,  p/\mathfrak p\sim  q/\mathfrak p
\,\,\, \mbox{ for each $\mathfrak p\in \prim(A)$.}
$$
 %
  Trivially,   $p\sim q$
implies  $p/\mathfrak p  \sim  q/\mathfrak p$ for all 
$\mathfrak p\in \prim(A)$.
Conversely, assuming
$p/\mathfrak p  \sim  q/\mathfrak p$ for all 
$\mathfrak p\in \prim(A)$, the continuity of the norm
ensures that  
a partial isometry connecting $p$
and  $q$  at a primitive ideal of $A$  can be 
lifted to a neighbourhood  $\mathcal N$, which
we may safely suppose to be clopen, because
 the topology of
$\prim(A)$ is boolean. Let $S$
be the set of   ideals $\mathfrak z \in   \prim(A)$ 
such that both $p/\mathfrak z$ and $q/\mathfrak z$ 
are nonzero. Since  $S$ is compact,  
a finite number of such clopen
neighbourhoods $\mathcal N_1,\ldots,\mathcal N_k$   
covers  $S.$  Without loss of generality,
 $\mathcal N_i\cap \mathcal N_j=\emptyset $  whenever   $i \not = j.$  
Adding up the associated partial isometries,  we obtain
  $p\sim q$. 
Our first claim is settled.

  \medskip
  \noindent
  {\it Claim 2:}
   For any  $p, q\in \proj(A)$ and $\mathfrak p\in \prim(A)$,\,\,\,     
$$
d_p(\mathfrak p)  =  d_q(\mathfrak p)
 \,\,\, \mbox{ iff }\,\,\,\,  p/\mathfrak p\sim  q/\mathfrak p.
$$

As we already know,  
 all primitive ideals of $A/\mathfrak p$ are maximal.
Moreover,  by \eqref{equation:uno},
 the finite-dimensional C*-algebra
 $A/\mathfrak p$  is  an isomorphic copy of
 the C*-algebra $M_d$ of  $d\times d$ 
complex matrices, with $d=\dim \pi$,
and $\pi$  
any irreducible representation with kernel $\mathfrak p.$
Thus $p/\mathfrak p\sim  q/\mathfrak p$ iff 
$\dim\,\, {\rm range}\,\, \pi(p)=
\dim\,\, {\rm range}\,\, \pi(q)$ iff
$d_p(\mathfrak p)=d_q(\mathfrak p)$.
Our second claim is settled.

\medskip
Claims 1 and 2 jointly show that
the map 
\begin{equation}
\label{equation:pretheta}
\theta\colon   [q]\in E(A)\mapsto d_q\in  \interval^{\prim (A)}
\end{equation} 
 is well defined. 
A direct inspection shows that $\theta$ is an
 isomorphism of $E(A)$ onto 
the  MV-algebra  $\theta(E(A))$ of  dimension
maps, with the pointwise operations of the standard
MV-algebra $\interval$,
\begin{equation}
\label{equation:theta}
\theta\colon  E(A) \,\,\,{\cong}\,\,\,
 \{\mbox{MV-algebra of  dimension maps on $\prim(A)$}\}.
\end{equation}

\smallskip
(II)  From
 \cite[Theorem 4.16]{mun11} 
we have  an isomorphism 
\begin{equation}
\label{equation:isomorphism}
^*\,\,\colon\,\, [q]\in E(A)\,\,\mapsto\,\, [q]^* \in 
 \interval^{\boldsymbol{\mu}(E(A))}
\end{equation}
 of   the
 semisimple MV-algebra
 $E(A)$ onto a {\it separating}  MV-algebra
of continuous $\interval$-valued functions over  the maximal spectral space
$\boldsymbol{\mu}(E(A))$. 
For every   $q\in \proj(A)$, the continuous function 
$[q]^*$ is  defined  by
the following stipulation:
For every maximal ideal $N \in \boldsymbol{\mu}(E(A))$,
\begin{equation}
\label{equation:quotient-value}
[q]^*(N)=
\mbox{ the only real number corresponding to }
\frac{[q]}{N}
\end{equation}
in the {\it unique} embedding of the simple MV-algebra $E(A)/N$ into
the standard MV-algebra  $\interval$.
Therefore, for all $\mathfrak p\in \prim(A)$, letting
$\eta(\mathfrak p)$ be the prime (automatically
 maximal) ideal of $E(A)$
corresponding to $\mathfrak p$ by Corollary \ref{corollary:spectral}(i),
we can write
\begin{equation}
\label{equation:lorena}
[q]^*(\eta(\mathfrak p))=
\mbox{ the only real number corresponding to }
\frac{[q]}{\eta(\mathfrak p)}.
\end{equation}
More generally, by \cite[Corollary 7.2.6]{cigdotmun}, for 
any  MV-algebra $B$
there is at most one embedding $\xi$ of   $B$
into the standard MV-algebra $\interval$.
  Thus whenever  such embedding 
  $\xi$ exists, 
  we may identify any $b\in B$ with the real
$\xi(b)\in \interval$    without fear of confusion.  
By    \cite[Corollary 3.5.4]{cigdotmun},
 for all $\mathfrak p\in \prim(A)$,
  the finite simple  MV-algebra $E(A/\mathfrak p)$,
as well as  its isomorphic copy
$E(A)/\eta(\mathfrak p)$, 
 are  uniquely embeddable onto a
subalgebra of the standard MV-algebra $\interval$.
Specifically,  let   $\pi$ be  an
irreducible representation
  of $A$ with $\ker \pi=
\mathfrak p$, and   
$
d=\dim \pi.
$
Let the MV-chain ${\mathsf L}_d$ be defined by  
$
{\mathsf L}_d=\left\{0,\frac{1}{d},\dots,\frac{d-1}{d},1\right\}.
$
Then from 
\begin{equation}
\label{equation:due}
E(A/\mathfrak p)\cong E(M_d)\cong
\mathsf L_d\subseteq \interval\,\,\,
\mbox{ and }\,\,\,
E(A)/\eta(\mathfrak p)\cong E(A/\mathfrak p)\cong \mathsf L_d
\subseteq \interval
\end{equation}
 we have unique embeddings of $E(A/\mathfrak p)$
 and $E(A)/\eta(\mathfrak p)$ into 
 $\mathsf L_d$.
For each $\mathfrak p\in \prim(A)$ and $q\in \proj(A)$,
recalling the definition of the isomorphism 
$E(A/\mathfrak p)\cong E(A)/\eta(\mathfrak p)$ in
Corollary \ref{corollary:spectral}(iii), 
we  can write

   \begin{eqnarray*}
 d_q(\mathfrak p)&=& \frac{\dim \range \,\, \pi(q)}
 {\dim \pi}
 \in {\mathsf L}_d,\,\,\,\mbox{ with } d=\dim \pi
 \,\,\,\mbox{ and }\ker\pi=\mathfrak p 
 \\[0.2cm]
{}&=&\mbox{the unique rational  in
${\mathsf L}_d$ corresponding to } 
[\pi(q)]  
\in E(M_d) 
\mbox{ by \eqref{equation:uno}} 
 \\[0.2cm]
{}&=&
\mbox{ the unique image 
in ${\mathsf L}_d$ of  } \left[\frac{q}{\mathfrak p}\right] \in
E(A/\mathfrak p), \,\,\, 
\mbox{ by \eqref{equation:due}}  
\\[0.2cm]
{}&=&
\mbox{ the unique image 
in ${\mathsf L}_d$ of  } 
 \frac{[q]}{\eta(\mathfrak p)}
\in \frac{E(A)}{\eta(\mathfrak p)}\,\,\,
\mbox{ in } {\mathsf L}_d, \mbox{ by \eqref{equation:due}}  
\\[0.2cm]
 {}&=&[q]^*(\eta(\mathfrak p)), \,\,\,
   \mbox{ by \eqref{equation:lorena}}.
  \end{eqnarray*}
The composite function
$d_q=[q]^*\circ\eta\colon \prim(A)\to \interval\cap\,\, \mathbb Q$ is continuous.
Since $E(A)^*$
  is hyperarchimedean,
  the range of $d_p$ is finite by 
   \cite[Lemma 4.6]{cigdubmun}. For
   all $q\in \proj(A)$,
    from \eqref{equation:pretheta}-\eqref{equation:theta}
    we get 
\begin{equation}
\label{equation:theta-bis}
   \theta([q])=d_q= [q]^*\circ\eta.
\end{equation}

   \bigskip
(III) First of all, the two separation  properties are
equivalent: for the nontrivial direction, let us
write   $d_r(\mathfrak p) < 
d_r(\mathfrak q)$.  Let $\mathcal O$ be
  an open interval such that 
  $d_r(\mathfrak p)< b <
d_r(\mathfrak q)$ for all $b\in \mathcal O$.   By  
 \cite[Lemma 3.1.9]{cigdotmun}, the free MV-algebra
$Free_1$ contains a McNaughton function
$\tau^*$ whose graph has three linear pieces as the
graph of  $\sigma^*$, with the additional property that
the non-constant linear piece of 
$\tau^*$ is nonzero only over  a nonempty open segment
contained in $\mathcal O.$ Then the composite function
 $\tau^*\circ d_r$ has value 0 at $\mathfrak p$
 and value 1 at  $\mathfrak q$.  Since  $\tau^*\circ d_r$
 is obtainable from $d_r$ by finitely many applications of
 the pointwise operations $\neg$ and $\oplus$, then
 $\tau^*\circ d_r$ is a  dimension map of $A$. 
 
 Having thus proved the equivalence of the two 
 separation properties,    the isomorphism  $^*$ in 
  \eqref{equation:isomorphism} maps $E(A)$ onto
the {\it separating}  MV-algebra
of continuous $\interval$-valued functions over  the maximal spectral space
$\boldsymbol{\mu}(E(A))$. By
\eqref{equation:isomorphism} and  \eqref{equation:theta-bis},
the MV-algebra of  dimension maps separates points. 
\end{proof}

\bigskip
In the light of 
Theorem \ref{theorem:equivalence-unified}, if
 $A$ satisfies  the two
equivalent conditions of Theorem \ref{theorem:equivalent-conditions},
identifying via $\eta$ the primitive ideal space
  $\prim(A)$ with  the maximal spectral space
$\boldsymbol{\mu}(E(A))$,
 we will henceforth  realize $E(A)$ as the
MV-algebra of  dimension maps  
\begin{equation}
\label{equation:ea}
\boxed{E(A)=\theta(E(A))=E(A)^*.}
\end{equation}

\medskip
\noindent
In particular, any  $f\in \Boole(E(A))$ will be identified
with a  $\{0,1\}$-valued  dimension map.

\begin{theorem}
\label{theorem:liminary}
Suppose the liminary C*-algebra
$A$ satisfies the two  equivalent conditions    of Theorem
\ref{theorem:equivalent-conditions}. 
We then have: 
 \begin{itemize}

 \smallskip
 \item[(i)] 
Every
finite subset of $E(A)$ generates a finite subalgebra
of $E(A).$ In other words, $E(A)$ is {\em locally finite}.

\smallskip
  \item[(ii)]    
  Every clopen   $W	\subseteq \prim(A)$
is the zeroset of some
$\{0,1\}$-valued  dimension map.

\smallskip
 \item[(iii)]  
 For each   dimension map   $d_p$  and rational     
 $\rho \in \interval$    there is a $\{0,1\}$-valued  dimension map      
 $b$    such that 
  $d_p^{-1}(r)=b^{-1}(0). $

%
%
%

\smallskip
\item[(iv)]  
Each  extremal state $s$ of $K_0(A)$ is {\em discrete}, in the sense
that $s(K_0(A))$ is a cyclic subgroup of $\mathbb R,\,\,\,\,$ 
\cite[p.70]{goo-ams}.

 \medskip
\item[(v)]
 $K_0(A)$ has general comparability, \cite[p.131]{goo-ams}.  
 
\end{itemize} 
\end{theorem}

\begin{proof}   
 (i) From
   \cite[Theorem 5.1(i)$\Leftrightarrow$(ii)]{cigdubmun},
in  view of
\eqref{equation:due}
 and 
 Corollary \ref{corollary:spectral}(i).

\medskip 
(ii)  Arguing as in  the proof of  Theorem
\ref{theorem:equivalence-unified}(III), for every
$x\in W $ some  dimension map  $r  \in E(A)$ vanishes
precisely  over
 a clopen neighbourhood
of  $x$ contained in $W$. 
Since   0 is isolated in the  range
of $r$,  replacing if necessary  $r$ by 
$$
r_W\,\,\,\,
=\underbrace{r \oplus\dots\oplus r}_{\mbox {\tiny suitably many summands}}
$$
we
may assume  $r_W$ to be $\{0,1\}$-valued. 
By  compactness, 
$W$ is covered by finitely many pairwise disjoint
clopens $W_1,\dots,W_m$ and corresponding 
$\{0,1\}$-valued  dimension maps
$r_{W_1},\dots,r_{W_m}$,  where for each
$i=1,\dots,m$,  the function  $r_{W_i}$ vanishes 
precisely over  $W_i.$
The zeroset of the  dimension map 
$r_{W_1}\wedge \dots\wedge  r_{W_m}$ 
coincides with $W$.   
 
 \medskip 
 (iii)   By Theorem
 \ref{theorem:equivalence-unified},
  the range of $d_p$ is finite and $d_p$
 is continuous. It follows that 
 $d_p^{-1}(\rho)$ is a clopen subset of $\prim(A).$
 Now apply (ii).

\bigskip
(iv) 
By Theorems  \ref{theorem:equivalent-conditions}
and \ref{theorem:multi}(iv)-(v), 
$K_0(A)$ is a lattice-ordered abelian group and
$E(A)=\Gamma(K_0(A))$.
%
%
%
By \cite[Theorem 12.18]{goo-ams}, the 
extremal states of  $K_0(A)$
coincide with the unit preserving $\ell$-ho\-m\-o\-mor\-ph\-isms of
$K_0(A)$ into the additive group  $\mathbb R$ of
 real numbers endowed with the usual order.
So let $s\colon K_0(A)\to \mathbb R$ be an extremal state.
The kernel of $s$  is a maximal ideal  
 of  $K_0(A)$. Corollary \ref{corollary:spectral}(v)
 yields a unique
 maximal ideal $\mathfrak s$ of $A$ such that
 $
\ker s = K_0(\mathfrak s).
 $
Since, as we have seen,
  (\cite[IV.15]{dav}, \cite[Corollary 9.2]{eff}),\,
 $K_0$ preserves exact sequences, then 
 $
 K_0(A/\mathfrak s)\cong {K_0(A)}/{\ker s}.
 $
 Again, Corollary \ref{corollary:spectral}(v)   yields
   a unique maximal ideal $M$
 of $E(A)$ such that
 $
 M=\eta(\mathfrak s)=K_0(A)\cap \Gamma(K_0(A)).
 $
We then have isomorphisms
\begin{equation}
\label{equation:quatuor}
E\left(\frac{A}{\mathfrak s}\right)
\cong
\frac{E(A)}{\eta(\mathfrak s)}
\cong 
\frac{\Gamma(K_0(A))}{K_0(A)\cap E(A)}
\cong
\frac{\Gamma(K_0(A))}{K_0(\mathfrak s)}
\cong
 \Gamma\left(\frac{K_0(A)}{\ker s}\right).
\end{equation}
By 
   (\cite[Definition 2.4]{mun-jfa}),
 $E(A)=\Gamma(K_0(A))$
coincides with the unit interval of $K_0(A)$
equipped with the order-unit 1 $=[1_A]$,
and with the operations  
$$x\oplus y= u\wedge (x+y)
\mbox{\,\,\,and\,\,\,}  \neg y=1-y.
$$
Since $M$ is a maximal ideal of $E(A)$,
by    \cite[Theorem 4.16]{mun11}, \,\,
 $E(A)/M\cong E(A/\mathfrak s)$  are
 uniquely isomorphic to
 the same  finite MV-subalgebra $L$ of $\interval$.
 By \cite[Corollary 3.5.4]{cigdotmun}, 
 $L = \{0,1/m,\dots(m-1)/m,1\}$ for a uniquely determined
 integer $m\geq 1.$
 Since $\Gamma$ is a categorical
 equivalence, from $\Gamma(\mathbb Z, m)= L$
 it follows that  
  $K_0(A/\mathfrak s)\cong (\mathbb Z, m)$,
showing  that  the state  
$s$ is discrete.

\bigskip
(v)  We prepare:

\bigskip
 \noindent
 {\it Claim 1:}
For any $p,q\in \proj(A)$ there are 
 clopens  $X, Y\subseteq \prim(A)$ = maximal ideal
 space of $A$,
such that for every $\mathfrak m\in \prim(A)$ 
$$
\mathfrak m\in X\,\,\,\, \Leftrightarrow 
\,\,\,\, d_p(\mathfrak m)\leq d_q(\mathfrak m) \mbox{ and }
\mathfrak m\in Y
\,\,\,\,    \Leftrightarrow 
\,\,\,\, d_p(\mathfrak m)\geq d_q(\mathfrak m).
$$

As a matter of fact, recalling the notational
stipulation \eqref{equation:odot},
by  \cite[Lemma 1.1.2]{cigdotmun},  
$d_p(\mathfrak m)\leq d_q(\mathfrak m)$
iff $(d_p\odot\neg d_q)(\mathfrak m)=0.$ Similarly,
and 
$d_p(\mathfrak m)\geq d_q(\mathfrak m)$
iff $ (d_q\odot\neg d_p)(\mathfrak m)=0.$
Now  the zeroset $f^{-1}(0)\subseteq 
\prim(A)$ of any  dimension map $f$  is clopen,
because the range of  $f$ is finite and $f$ is continuous.
Conversely, by (iii), every clopen subset of 
$\prim(A)$
 is the zeroset of some  $\{0,1\}$-valued  dimension map
 $f$, i.e.,
 (Theorem 
\ref{theorem:threeconditions}(ii)$\Leftrightarrow$(iv)),
 the zeroset of some characteristic
element of $K_0(A)$.
  Our first claim is settled.
  
  \bigskip
  
A routine variant of the proof of Claim 1 yields:
  
  \bigskip
 \noindent
 {\it Claim 2:}
 For any $p,q\in \proj(A)$ there are 
 clopens  $X, Y\subseteq \prim(A)$
such that 
for every $\mathfrak m\in \prim(A)$,
$$
\mathfrak m\in X \Leftrightarrow d_p(\mathfrak m) < d_q(\mathfrak m) \mbox{ and }
\mathfrak m\in Y   \Leftrightarrow d_p(\mathfrak m)\geq d_q(\mathfrak m).
$$

 \medskip
 \noindent
 Next let
 $$
  \maxspec(K_0(A))
 $$
 denote the maximal spectral space of $K_0(A),$
 By
 \cite[Theorems 7.2.2, 
 7.2.4]{cigdotmun} and Corollary \ref{corollary:spectral}(ii),
 \begin{equation}
 \label{equation:trespazi}
 \maxspec(K_0(A))\cong \boldsymbol{\mu}(E(A))
 \cong \prim(A)
 \end{equation}
and   $ \maxspec(K_0(A))$  can be safely identified with
 $\prim(A)$ and with $\boldsymbol{\mu}(E(A))$.
By \cite[Theorems 3.8-3.9]{mun-jfa},
 $K_0(A)$ is (isomorphic to)
 the
unital $\ell$-group of functions on  $\prim(A)$
generated by the  dimension maps,
 with the constant $u=1=[1_A]$ as the
unit,  and with the pointwise $\ell$-group operations
of $\mathbb R.$ By Theorem
\ref{theorem:equivalence-unified},  each function in  $K_0(A)$ is
continuous, rational-valued, and has a  finite range.

 As explained
  in  \cite[p.126]{goo-ams},
  to prove that $K_0(A)$ has general comparability,
  for all $h,k\in K_0(A)$ we must find a direct product decomposition
  $$
  K_0(A)=G_1\times G_2
  $$
  such that the $G_1$-components of $h$ and $k$ satisfy
  $h_1\leq k_1$ while the $G_2$-components of 
  $h$ and $k$ satisfy  $h_2\geq k_2.$  
  By the translation invariance of the lattice 
  order of  $ K_0(A)$
   and the defining property of
  the  unit $u$ of $K_0(A)$,
replacing, if necessary,   $h,k$ by $h+mu, k+mu$
(for a suitably large integer $m$), we may   assume
   $h,k \geq 0.$
   
     \bigskip
 \noindent
 {\it Claim 3:}
There is  a clopen
  $X_1\subseteq  \maxspec(K_0(A))$ coinciding with 
  the set  $X_1$ of maximal
  ideals of   $N$ of $K_0(A)$ such that 
  $h/N\leq k/N$.    
 
As a matter of fact, let
  $$x_1,x_2,\dots, x_{n_1}\,\,\, \mbox{ and }\,\,\,\,
  y_1,y_2,\dots, y_{n_2}$$
be elements  of $\Gamma(K_0(A))=E(A)$
having the following properties:
$$x_i\oplus x_{i+1}=x_i,\,\,\,\, \,\,\,\, 
 \sum_{i=1}^{n_1}x_i=h,  \,\,\,\,\,\,\,\,
 y_i\oplus y_{i+1}=y_i,\,\,\,\,\,\,\,
 \sum_{i=1}^{n_2}y_i=k.
 $$
 Their existence is ensured by
 \cite[Proposition 3.1(i)]{mun-jfa}.  (Actually, 
 these sequences are uniquely determined by $h$ and $k$, up to
 a tail of zeros.)
Adding a finite tail of zeros
 to the shortest
sequence,
we may assume   $n_1=n_2=n$ without loss of
generality.
  Recalling the notational stipulation
  \eqref{equation:odot}, for 
  each $i=1,\dots,n$ let
   $X_{1,i}$ be the zeroset of  the dimension map
   $x_i\odot \neg y_i$. The identification
   \eqref{equation:trespazi} yields
   $$
   X_{1,i}=\{N\in \maxspec(K_0(A))\mid
   x_i/N\leq y_i/N\,\,\,\mbox{for all }\,\,\, i=1,\dots,n)\},
   $$
   because  (\cite[Lemma 1.1.2]{cigdotmun}),
 $$x_i/N\odot\neg y_i/N =0\,\,\,\,\, \Leftrightarrow\,\,\,\,\,
 x_i/N\leq y_i/N.
 $$
Now by
  \cite[Proposition 3.1(ii)]{mun-jfa},  
   for any $N \in \maxspec(K_0(A))$
the inequality
$h/N\leq k/N$ is equivalent to the simultaneous occurrence
of the inequalities 
$$
x_1/N\leq y_1/N,\,\dots,\, x_n/N\leq y_n/N.
$$
As a consequence,   the set 
 $X_1=\bigcap_{i=1}^n X_{1,i}$ satisfies
\begin{equation}
\label{equation:exone}
  N \in X_1\Leftrightarrow h/N\leq k/N.
\end{equation}
Since the range of every  dimension map $f$
is finite  and $f$ is continuous,   
each $X_{1,i}$
 is a clopen subset of $ \maxspec(K_0(A))$, and so is
 $X_1$. 
 Thus  $X_1$
 has the desired properties, and our third claim is settled.

\medskip 
The complementary  clopen  
  $X_2=  \maxspec(K_0(A))\setminus X_1$  has the
  property that   for
    every maximal ideal $N$ of $K_0(A)$,\,\,\,
$
 N \in X_2\Leftrightarrow h/N >  k/N.
$

\medskip
In view of (iii),  for each  $j=1,2$ 
 let $e_j $ be the uniquely determined
  $\{0,1\}$-valued  dimension map
satisfying
  $e_j^{-1}(0)=X_j.$  Each $e_j$ is a 
 characteristic element of  $K_0(A)$,
(Theorem \ref{theorem:threeconditions}).
  Let $I_j$ be the ideal of $K_0(A)$ generated 
  by   $e_j$.   
  The $\ell$-homomorphisms of $K_0(A)$ into itself induced by the
  two ideals $I_1, I_2$ 
    provide the desired direct product decomposition
    $
K_0(A)\cong  K_0(A)/I_1 \times K_0(A)/I_2.
$
Up to isomorphism, every  
  $g \in K_0(A)$ splits  into its restrictions
   $g_1= g  \restrict X_1$
and 
  $g_2= g \restrict X_2$.  The $ K_0(A)/I_1$
  components of $h$ and $k$ satisfy  $h_1\leq k_1$.
  The $ K_0(A)/I_2$ components satisfy
  $h_2 >  k_2$.
A fortiori,  $K_0(A)$ has general comparability. 
 \end{proof}

\medskip

\subsection*{Central projections as fixpoints}

\bigskip 
\begin{corollary} 
\label{corollary:sharpener}
Suppose the liminary C*-algebra
$A$ satisfies the two equivalent conditions     of Theorem
\ref{theorem:equivalent-conditions}.  
For all $p,q\in \proj(A)$ we have:

\smallskip
\begin{enumerate}
\item  (Fixpoint) The sequence
$
[p]\sqsupseteq [p]_\sigma \sqsupseteq [p]_{\sigma\circ \sigma}
\sqsupseteq [p]_{\sigma\circ \sigma\circ \sigma}
\sqsupseteq \dots
$
is eventually constant.  

\medskip
\item Let \,\,$n(p)$\, be the least integer $m$ such that
$$
 [p]_{\underbrace{\sigma\circ \dots \circ \sigma}_{m\, \rm times}}
 =
  [p]_{\underbrace{\sigma\circ \dots \circ \sigma}_{m+1\, \rm times}}.
$$  
Then
\begin{itemize}

\smallskip
\item[(i)]   If $p$ is central or 
$p\sim 1_A-p$,\,\, 
$n(p)=0$.

\medskip
\item[(ii)]    If $n(p)=0$ then  for every $\mathfrak p\in \prim(A)$,\,
either 
 $p/\mathfrak p\in \{0,1\}$ or   $p/\mathfrak p\sim (1_A-p)/\mathfrak p$.
 
\medskip
\item[(iii)]  
Suppose for each $\mathfrak p\in \prim(A)$,
  $$[p]/\mathfrak p\sim[1_A-p]/\mathfrak p
  \Leftrightarrow [q]/\mathfrak p\sim[1-q]/\mathfrak p.$$ 
 %
 %
 %
  Then $[p] \sqsubseteq [q]  \Rightarrow n(p) \leq n(q).$ 
\end{itemize}

\bigskip
\item
Let the set $\mathcal C_p\subseteq E(A)$ be defined by
 $$
 \mathcal C_p=
 \{[r]\in E(A)\mid r  \mbox{ is a central projection in $A$ such that }
 [r] \sqsubseteq [p] \}.
 $$
 
\smallskip
 \begin{itemize}
 \item[(I)]
 $\mathcal C_p$ nonempty. 

\medskip
 \item[(II)]
 $\mathcal C_p$ is a singleton
iff for no $\mathfrak p\in \prim(A)$ we have $p/\mathfrak p
\sim (1_A-p)/\mathfrak p$.  

\medskip
 \item[(III)]
 When
 $\mathcal C_p$ is a singleton,  the 
unique element $[r]\in \mathcal C_p$ equals  the fixpoint
$ [p]_{\underbrace{\sigma\circ \dots \circ \sigma}_{n(p)\, \rm times}}.$

\medskip
 \item[(IV)]
If 
 $p/\mathfrak p
\sim (1_A-p)/\mathfrak p$
for  some $\mathfrak p\in \prim(A)$ then 
$$ [p]_{\underbrace{\sigma\circ \dots \circ \sigma}_{n(p)\, \rm times}}
\,\,\,\oplus\,\,\,
[p]_{\underbrace{\sigma\circ \dots \circ \sigma}_{n(p)\, \rm times}}
\,\,\,\,\,\mbox{\,\,\,\,and\,\,\,}\,\,\,\,\,
  [p]_{\underbrace{\sigma\circ \dots \circ \sigma}_{n(p)\, \rm times}}
\,\,\,\odot\,\,\,
[p]_{\underbrace{\sigma\circ \dots \circ \sigma}_{n(p)\, \rm times}}
$$
are two distinct elements of  $ \mathcal C_p$. 
\end{itemize}
\end{enumerate}
\end{corollary}

\begin{proof} In view of
Theorem
\ref{theorem:equivalence-unified}, throughout we will
identify
 $E(A)$ with the separating MV-algebra   
of all  dimension maps.
Thus each  $[p]\in E(A)$  is a continuous 
rational-valued function $d_p$ with a finite range
over the boolean space  
${\boldsymbol{\mu}}(E(A))
=\Spec(E(A))
\cong \prim(A).$ 

\medskip

 (1)  The value 
$1/2$ is 
isolated in $\range(d_p).$ 
The definition of the map 
$\sigma^*\colon \interval\to \interval$
immediately yields the desired conclusion.

\medskip
To prove  (2) we argue as follows:

\smallskip

 (i)
 If $p$ is central, $\range(d_p) \subseteq \{0,1\}$, 
whence  $n(p)=0$, because $\sigma^*(0)=0$ and
 $\sigma^*(1)=1$. 
  If $p\sim 1_A-p$  
 then  $d_p(\mathfrak p)=1/2$ 
for all  $\mathfrak p\in \prim(A)$.  
From  $\sigma^*(1/2)=1/2$ it follows that
$n(p)=0$.

\medskip
(ii)   
   If $n(p)=0$ then  for every $\mathfrak p\in \prim(A)$,\,\,\,
 $p/\mathfrak p\in \{0,1\}$ or   $p/\mathfrak p\sim (1_A-p)/\mathfrak p$,
 because  $\sigma^*(t)=t$ iff $t\in\{0,1/2,1\}.$

\medskip
(iii)  The hypothesis means
  $d_p^{-1}(1/2)=d_q^{-1}(1/2)$.  
  The conclusion then follows
by  definition of $\sqsubseteq$ and $n(p).$

 \medskip
 (3)  
 Let $f_p =  [p]_{{\sigma\circ \dots \circ \sigma}} =
 (d_p)_{{\sigma\circ \dots \circ \sigma}}$\,\,\,\,\,
 ($n(p)$ times).
  Let  $C\subseteq\prim(A)$
   be defined by  $C= f_p^{-1}(1/2)$.
   By Theorem
\ref{theorem:equivalence-unified}, $C$
is clopen.
  
    \smallskip
  If 
$C=\emptyset$, then  $f_p=(f_p)_\sigma$ by definition
of $n(p)$.
For every  $\mathfrak p\in \prim(A)$,  either
$f_p(\mathfrak p)<1/2$, in which case 
$f_p(\mathfrak p)=0$,
or else $f_p(\mathfrak p)>1/2$,  in which case
$f_p(\mathfrak p)=1$.
Thus $f_p$ is the only element of  
$\mathcal C_p$.  This proves (I)-(IV) for the present case.

\medskip
If   $C\not=\emptyset$,
then $f_p\odot f_p$ pushes  the graph of $f_p\restrict C$ down to  0, leaving unaltered the rest of $f_p$\,; evidently
 $f_p\odot f_p$ is a $\{0,1\}$-valued  
 dimension map and 
$f_p\odot f_p=(f_p\odot f_p)_\sigma$.
 Similarly,  $f_p\oplus  f_p$
 pushes  the graph of  $f_p\restrict C$ up  to  1,
  leaving  the rest  unaltered.
So $f_p\oplus f_p=(f_p\oplus f_p)_\sigma$.
This completes the proof of (I)-(IV). 
\end{proof}

Intuitively,  
the map $[p]\mapsto [p]^\Game= [p]_\sigma$ is ``centripetal''
in the sense that   $ [p]^\Game\sqsubseteq [p] $, and
  a finite number of iterations of the map leads
to a unique $\Game$-fixpoint, in such a way that if 
 $p$ is central then $[p]=[p]^\Game$. 
If $p$ is not central and 
$p/\mathfrak p\not=(1_A-p)/\mathfrak p$
for every  primitive
ideal  $\mathfrak p$ of $A$, then the
$\Game$-fixpoint   $[q]$ of
$[p]$  arises from some central
projection $q$ of $A$. 
If  
$p/\mathfrak p =(1_A-p)/\mathfrak p$
for some primitive
ideal  $\mathfrak p$ of $A$, then the same 
holds (not for $[q]$, but)  for  $[q]\oplus [q]$,
or  for $[q]\odot [q]$.

\bigskip
\section{Concluding remarks}
 AF$\ell$-algebras include
many interesting 
   classes of AF algebras, well beyond the trivial  examples
of  commutative AF algebras and finite-dimensional C$^*$-algebras.
Nontrivial examples are given by
the CAR algebra and, more generally,
 Glimm's UHF algebras, \cite{eff, brarob},
 the
Effros-Shen C*-algebras   
$\mathfrak F_{\theta}$ for irrational $\theta\in\interval$, 
\cite[p.65]{eff}, which play  an
interesting role in topological dynamics,
\cite{bratteli1, bratteli2, pimvoi}.
Further examples are provided by
AF$\ell$-algebras whose $K_0$-group has
general comparability, \cite[Proposition 8.9, p.131]{goo-ams}.
Non-simple examples include the  Behncke-Leptin
C*-algebras $\mathcal A_{m,n}$ with a two-point dual
\cite{behlep},   
and  AF algebras with a directed set of finite
dimensional *-subalgebras, \cite{laz2}.

The  ``universal''
  AF algebra $\mathfrak M$
of  \cite[\S 8]{mun-jfa} is an AF$\ell$-algebra.
It is defined by
$$
E(\mathfrak M)=Free_\omega = \mbox{ the free
countably generated  MV-algebra.}
$$
Every  AF algebra  with comparability
of projections is a quotient of $\mathfrak M$
by a   primitive essential ideal,
 \cite[Corollary 8.7]{mun-jfa}. Every 
 (possibly non-unital) AF algebra may be embedded into
a quotient of $\mathfrak M$, \cite[Remark 8.9]{mun-jfa}.
One more example is given by the 
Farey  AF$\ell$-algebra $\mathfrak M_{1}$
introduced in 
 \cite{mun-adv}. 
 It is defined by
$
E(\mathfrak M_1)=Free_1.
$
By \cite{pimvoi}, 
 every irrational rotation C$^*$-algebra 
 is embeddable into some (Effros-Shen) simple quotient of  
 $\mathfrak M_1, $ \cite[Theorem 3.1 (ii)]{mun-adv}.
$\mathfrak M_1$, in turn,
  is embeddable into
  Glimm's universal UHF algebra,
   \cite[Theorem 1.5]{mun-milan}.
 As shown in \cite{mun-lincei, mun-milan}, 
 the AF algebra $\mathfrak A$ more recently
 considered by   Boca \cite{boc}  coincides with 
$\mathfrak M_1$. For an account of the interesting properties and  
 applications of $\mathfrak M_1$  see 
 \cite{boc, eck,
 mun-adv,  mun-lincei, 
	mun-milan}. 
 
 Liminary   $C^*$-algebras with
boolean spectrum, and more generally with Hausdorff spectrum,
 are considered
by Dixmier \cite[passim]{dix}.
 Liminary   $C^*$-algebras with
boolean spectrum
{\it per se} are the main topic of \cite{cigellmun}. Here
 the authors consider the analogue of 
Kaplansky's problem for these algebras,
and prove that the Murray von Neumann
order of projections alone is 
sufficient to uniquely recover
the $C^*$-algebraic structure.
Thus
{\it  If two liminary C*-algebras with Boolean spectrum have order-isomorphic Murray von Neumann posets,   then they are isomorphic.
}

  \begin{problem} 
Extend  the characterization
 $$
\boxed{\mbox{$p$ central in $A$  $\Leftrightarrow$ $[p] \sqsubseteq$-minimal 
 $\Leftrightarrow$ $[p]$ characteristic
in $K_0(A)$
$\Leftrightarrow$ $[p]$ a fixpoint}}
$$
outside the domain of  AF$\ell$-algebras. 
  \end{problem}

 \bibliographystyle{plain}

\end{document}